\documentclass[a4paper,10pt]{article}
\usepackage{amsmath,amssymb,latexsym,amsthm}
\usepackage{array,graphicx,epsfig,xcolor}
\usepackage{dsfont}
\usepackage{tikz}
\usepackage{mathrsfs}
\usepackage{stmaryrd}
\usepackage[english]{babel}
\usepackage{enumerate}
\def\R{{\mathbb R}}
\def\N{{\mathbb N}}

\def\Lip{\mathcal Lip}

\def\ds{\displaystyle}

\newcommand{\norm}[1]{\left\Vert#1\right\Vert}

\newtheorem{theorem}{Theorem}[section]
\newtheorem{definition}[theorem]{Definition}
\newtheorem{corollary}[theorem]{Corollary}

\newtheorem{lemma}[theorem]{Lemma}
\newtheorem{remark}[theorem]{Remark}

\numberwithin{equation}{section}

\title{Dissipative boundary conditions for $2\times 2$ hyperbolic systems of conservation laws for entropy solutions in BV\thanks{This work was
supported by the ERC advanced grant 266907 (CPDENL) of the 7th Research Framework
Programme (FP7). The authors acknowledge the support of the UMI IFCAM (Indo-French Center for Applied Mathematics).}}
\author{Jean-Michel Coron\thanks{Sorbonne Universit\'{e}s, UPMC Univ Paris 06, Laboratoire
Jacques-Louis Lions, CNRS UMR 7598,  4 place Jussieu, 75252 Paris cedex 05, France. Email: \tt{coron@ann.jussieu.fr}.}
\and
Sylvain Ervedoza\thanks{Universit\'{e} Paul Sabatier, Institut de Math\'{e}matiques de Toulouse, CNRS UMR 5219, 118 route de Narbonne
31062 Toulouse Cedex 9, France. Email: \tt{sylvain.ervedoza@math.univ-toulouse.fr}.}
\and
Shyam Sundar Ghoshal\thanks{Gran Sasso Science Institute (GSSI), viale Francesco Crispi 7, 67100 L'Aquila, Italy. Email: \tt{shyam.ghoshal@gssi.infn.it}.}
\and
Olivier Glass\thanks{Universit\'{e} Paris-Dauphine, Ceremade, CNRS UMR 7534, Place du Mar\'{e}chal de Lattre de Tassigny, 75775 Paris Cedex
16, France. Email: \tt{glass@ceremade.dauphine.fr}.}
\and
Vincent Perrollaz\thanks{Universit\'{e} Fran\c{c}ois Rabelais, Laboratoire de Math\'{e}matiques et Physique Th\'{e}orique, CNRS UMR 7350, Parc de Grandmont,
37200 Tours, France. Email: \tt{vincent.perrollaz@lmpt.univ-tours.fr}.}}
\date{}
\begin{document}

\maketitle

\begin{abstract}
	In this article, we investigate the BV stability of $2 \times 2$ hyperbolic systems of conservation laws with strictly positive velocities under dissipative boundary conditions. More precisely, we derive sufficient conditions guaranteeing the exponential stability of the system under consideration for entropy solutions in BV.  Our proof is based on a front tracking algorithm used to construct approximate piecewise constants solutions whose BV norms are controlled through a Lyapunov functional. This Lyapunov functional is inspired by the one proposed in J. Glimm's seminal work  \cite{Glimm65}, modified with some suitable weights in the spirit of the previous works \cite{Coron-Andrea-Bastin-IEEE-2007,Coron-Bastin-Andrea-SICON-2008}.
\end{abstract}

\tableofcontents
%
%
%
%
%
%
%
%
\section{Introduction}
%
%
%
%
\subsection{Setting and main result}
The goal of this article is to study the exponential stability of $2\times 2$ systems
of conservation laws on a finite interval, by means of boundary feedbacks, in the context of weak entropy solutions. To be more precise, we consider the following setting: Let $\Omega$ be an open subset of $ \mathbb{R}^2$ with $0 \in \Omega$, and $f:\Omega\mapsto \R^2$ be a smooth function (supposed to satisfy the strict hyperbolicity condition described below) and consider the following system of two conservation laws
\begin{equation}
	\label{System-U}
		 \partial_t u+ \partial_x (f(u)) = 0  \quad \mbox{for}\ (t,x) \in (0, \infty)\times (0,L).
\end{equation}
In System \eqref{System-U}, the solution $u= u(t,x) = (u_1, u_2)^T$ has $2$ components and the space variable $x$ belongs to the finite interval $(0,L)$. 
We assume the flux function $f$ to satisfy the strictly hyperbolicity conditions, that is,
\begin{equation} \label{StrHyp}
\forall u \in \Omega, \ \text{ the matrix } A(u)=Df(u) \text{ has 2 real distinct eigenvalues } \lambda_1(u)<\lambda_2(u).
 \end{equation}
Furthermore, we make the assumption that both velocities are positive, so that we finally get:
\begin{equation} \label{PositiveVel}
 0 < \lambda_1(u) < \lambda_2(u)\qquad \text{for all}\ 	 u \in \Omega.
\end{equation}
Note that the case of two strictly negative velocities is obviously equivalent by the change of variable $x \to L-x$.
\\
System \eqref{System-U} is completed with boundary conditions of the form
\begin{equation} \label{BoundaryConditions}
u(t,0) = K u(t,L), 	
\end{equation}
where  $K$ is a $2\times 2$ (real) matrix. Here, for sake of simplicity, we assume that the boundary condition is linear but other nonlinear forms could be considered.
\par
\ \par
Clearly $u \equiv 0 $  is an equilibrium of \eqref{System-U}-\eqref{BoundaryConditions}. We are interested in the exponential stability
of this equilibrium in the $BV$ space for entropy solutions. We recall that the space $BV$ is natural for solutions of nonlinear hyperbolic systems of conservation laws, and is in particular the space considered in the celebrated paper by J. Glimm \cite{Glimm65}. \par
\ \par
In order to discuss the condition that we impose on $K$, we further introduce the left and right eigenvectors of $A(u) = Df(u)$: for each $k=1,2$, we define $r_k(u)$ as a right eigenvector of $A(u)$ corresponding to the eigenvalue $\lambda_k(u)$:
\begin{equation}
	\label{Right-Eigenvector}
	A(u) r_k(u)=\lambda_k(u) r_k(u), \,r_k(u)\not =0, \qquad \ k=1,2,\ u\in \Omega.
\end{equation}
We also introduce correspondingly left eigenvectors $\ell_k(u)$ of $A(u)$
\begin{equation} \label{Left-Eigenvector}
	\ell_k(u) A(u) = \lambda_k(u) \ell_k(u), \quad \hbox{ with }
	\ell_k(u) \cdot r_{k'} (u) =
			\left\{ \begin{array}{l}	
			1 \hbox{ if } k = k',
				\\
				0 \hbox{ if } k \neq k'.
			\end{array}\right.
\end{equation}
We further impose that the hyperbolic system \eqref{System-U} is genuinely non-linear in the sense of Lax \cite{Lax}, i.e.
$$
	 D\lambda_k(u)\cdot r_k(u) \neq 0 \ \mbox{for all} \ u\in\Omega.	
$$
Changing the sign of $r_k(u)$ and $\ell_k(u)$ if necessary, we can therefore assume
\begin{equation} \label{Genuine-Non-Linear}
D \lambda_k(u)\cdot r_k(u) > 0 \ \mbox{for all} \ u\in\Omega.		
\end{equation}
\par
As the total variation $\mbox{TV}_{[0,L]}$ is only a semi-norm on $BV(0,L)$ (it vanishes for constant maps), it
is convenient to define the following norm on $BV(0,L)$ as
\begin{equation}
\label{BV-normdef}
|u|_{BV} :=  \mbox{TV}_{[0,L]}( u) +\int_0^L|u(x)|dx,\qquad u\in BV(0,L).
\end{equation}
It is useful to recall that a function $u \in BV(0,L)$ has at most countably many discontinuities and have at each point left and right limits. In particular one can define $u(0^{+})$ and $u(L^{-})$ without ambiguity. \par

Now we recall that entropy solutions are weak solutions of \eqref{System-U} in the sense of distributions, which satisfy moreover {\it entropy conditions} for the sake of uniqueness. A way to express these entropy conditions consists in introducing entropy/entropy flux couples for \eqref{System-U} as any couple of regular functions $(\eta,q) : \Omega \rightarrow \R$ satisfying:
\begin{equation}
\label{Def:CoupleEntropie}
\forall U \in \Omega, \ \ D\eta(U) \cdot Df(U) = Dq(U).
\end{equation}
Of course $(\eta,q)=(\pm \mbox{Id}, \pm f)$ are entropy/entropy flux couples. Then we have the following definition (see \cite{Bressan-Book-2000,Dafermos,Lax}):
\begin{definition} \label{Def:SolutionEntropie}
A function $u \in L^\infty(0,T;BV(0,L)) \cap \Lip(0,T ;L^1(0,L))$ is called an {\it entropy solution} of \eqref{System-U} when,
for any entropy/entropy flux couple $(\eta,q)$, with $\eta$ convex, one has in the sense of measures
\begin{equation} \label{Def:SolEntrop1}
\eta(u)_t + q(u)_x \leq 0,
\end{equation}
that is, for all $\varphi \in {\mathcal D}((0,T) \times (0,L))$ with $\varphi \geq 0$, 
\begin{equation} \label{InegEntropie}
\int_{(0,T) \times (0,L)} \big( \eta(u(t,x)) \varphi_t(t,x) + q(u(t,x)) \varphi_x(t,x) \big) \, dx \, dt \geq 0.
\end{equation}
\end{definition}
These entropy conditions can be justified for instance by vanishing viscosity and are automatically satisfied by classical (by which we mean of class $C^{1}$) solutions. \par
\ \par
Our main result is the following one:
\begin{theorem} \label{Thm-Main}
Let the system \eqref{System-U} be strictly hyperbolic and genuinely nonlinear in the sense of \eqref{Genuine-Non-Linear}, and assume that the velocities are positive in the sense of \eqref{PositiveVel}.
	
If the matrix $K$ satisfies
\begin{multline}
\label{Assumption-On-K-new}
\inf_{\alpha \in (0,+\infty)} \left( \max \big\{|\ell_1(0) \cdot K r_1(0) | + \alpha |\ell_2(0) \cdot K r_1(0)  |, \right.  \\
\left.  \alpha^{-1}|\ell_1(0)\cdot K r_2(0) | +|\ell_2(0) \cdot K r_2(0) |\big\}\right)  	< 1,
\end{multline}
then there exist positive constants $C$, $\nu$, $\varepsilon_0 >0$, such that for every $u_0 \in BV(0,L)$ satisfying
\begin{equation} \label{Assumption-On-U-0}
|u_0|_{BV} \leq \varepsilon_0,
\end{equation}
there exists an entropy solution $u$ of \eqref{System-U} in $L^\infty(0,\infty; BV(0,L))$
satisfying $u(0,\cdot)=u_0(\cdot)$ and \eqref{BoundaryConditions} for almost all times,  such that
\begin{equation} \label{Exp-Stabilization-Estimate}
|u(t)|_{BV} \leq C \exp(-\nu t) |u_0|_{BV}, \qquad t \geq 0 .
\end{equation}
\end{theorem}
In other words, Theorem \ref{Thm-Main} states that \eqref{Assumption-On-K-new} is a sufficient condition for the (local) exponential stability of $u\equiv 0 $ with respect to the $BV$-norm for \eqref{System-U} with the boundary condition \eqref{BoundaryConditions}.
\begin{remark}
In Theorem~\ref{Thm-Main}, there is no claim of uniqueness of the solution of \eqref{System-U} with initial condition $u(0,\cdot)=u_0(\cdot)$ and boundary law \eqref{BoundaryConditions}. The initial-boundary value problem for hyperbolic systems of conservation laws is a very delicate matter, even in the non-characteristic case (when no characteristic speed vanishes). We refer for instance to Amadori \cite{Amadori}, Amadori-Colombo \cite{Amadori-Colombo-1997,Amadori-Colombo-1998}, Colombo-Guerra \cite{Colombo-Guerra-2010}, Donadello-Marson \cite{Donadello-Marson-2007} and references therein. In particular, in the open loop case, it is possible to construct a {\it Standard Riemann Semigroup} as a limit of front-tracking approximations (which is also the construction method employed here). However, up to our knowledge, these results do not quite cover our feedback boundary condition \eqref{BoundaryConditions}, and no result of uniqueness in the same spirit as \cite[Theorem 9.4]{Bressan-Book-2000} is available in our situation.
\end{remark}
Our approach will be based on the construction of a Lyapunov functional for solutions $u$ of \eqref{System-U}-\eqref{BoundaryConditions},
in the spirit of earlier works by B. d'Andr\'ea-Novel,
G. Bastin and J.-M. Coron \cite{Coron-Andrea-Bastin-IEEE-2007,Coron-Bastin-Andrea-SICON-2008}.
The main difference with these works is that they consider Lyapunov functionals estimating
the $H^2$ norm of the solutions, which are classical solutions and therefore forbids the presence of shocks.
More recently in \cite{Coron-Bastin-2014}, a Lyapunov approach was
also discussed to derive local stabilization results in $C^1$. But we will rather design a Lyapunov functional that measures the $BV$-norm of the solution.

Our Lyapunov functional is actually inspired by the classical functional introduced
by J. Glimm in \cite{Glimm65}, which measures the strength of interaction and a quadratic
quantity measuring the ``potential for future interactions''. In this celebrated paper, this functional
was devoted to prove the existence of global solutions of hyperbolic systems in a small $BV$-neighborhood of a constant map.
While the construction in \cite{Glimm65} is applied to an approximating sequence generated by means of the so-called random choice method, we will follow the method of wave-front-tracking, originally introduced by C. Dafermos \cite{Dafermos:FT} in the context of scalar conservation laws, and extended in the context of $2 \times 2$ hyperbolic systems of conservation laws (as considered in this paper) by R.~DiPerna in \cite{DiPerna-1976}. This method is at the core of the book \cite{Bressan-Book-2000} by A. Bressan to which we will refer several times in the following.
%
%
%
%
%
%
%
%
\subsection{Previous results}
{\bf Classical setting.}
We first recall some previous results on the exponential stability of $n\times n$ hyperbolic systems
on the finite interval $(0,L)$ in the ``classical case'', when no entropy conditions are needed. \par
For simplicity of this presentation, we assume that all the characteristic speeds are positive.
Our dynamical system takes the form
\begin{equation} \label{gen-hyp}
\left\{ \begin{array}{ll}
u_t+A(u)u_x=0& \quad \mbox{for}\ (t,x) \in (0, \infty)\times (0,L), \\
u(t,0)=G(u(t,L))& \quad \mbox{for}\ t \in (0,+\infty),
\end{array} \right.
\end{equation}
where $A:\Omega\to \R^{n\times n}$ and $G:\Omega\to \R^n$ are smooth maps such that
\begin{gather}
\label{property-A}
\text{for every $u\in \Omega$, the eigenvalues of $A(u)$ are real, positive and distinct,} \\
\label{property-G}
G(0)=0.
\end{gather}
It follows from \eqref{property-G} that $u\equiv0$ is an equilibrium of our dynamical system \eqref{gen-hyp}.
Recall that one says that this equilibrium
is exponentially stable for \eqref{gen-hyp} with respect to the norm $|\cdot|_X$ on a linear space $X$ of
functions from $[0,L]$ into $\R^n$ if there exist three positive constants $C$, $\nu$ and  $\varepsilon_0 >0$ such that,  for every  initial data $u_0 \in X$ satisfying the compatibility conditions adapted to $X$ and such that
\begin{equation}\label{Assumption-On-U-0-X}
|u_0|_X\leq \varepsilon_0,
\end{equation}
the solution $u$ of \eqref{gen-hyp} associated to the initial condition $u(0,\cdot)=u_0(\cdot)$ is such that
\begin{equation} \label{Exp-Stabilization-Estimate-X}
|u(t)|_X \leq C \exp(-\nu t) |u_0|_X \qquad \text{for}\ t \geq 0 .
\end{equation}
Again to simplify the presentation, we moreover assume that
\begin{equation}\label{Adiag}
  A(0) \text{ is a diagonal matrix.}
\end{equation}
In fact, this can be assumed without loss of generality by performing a linear change of variables on $u$ if necessary.
Let
\begin{equation}\label{defK}
  K:=G'(0)\in \R^{n\times n}.
\end{equation}
%
%
%
\noindent
{\it The linear case.}
The first studies on the exponential stability of $0$ concern naturally the linear case:
\begin{equation} \label{caslinear}
	A(u)=\Lambda =\text{diag}(\lambda_1,\ldots, \lambda_n), \quad G(u)=Ku \qquad \text{for}\ u \in \Omega.
\end{equation}
Hence the dynamical system \eqref{gen-hyp} is now
\begin{equation}\label{sys-lin}
\left\{\begin{array}{ll}
u_t+\Lambda u_x=0& \quad \mbox{for}\ (t,x) \in (0, \infty)\times (0,L), \\
u(t,0)=Ku(t,L) & \quad \mbox{for}\ t\in (0,+\infty).
\end{array} \right.
\end{equation}
Regarding $X$, we can consider various functional settings which give analogous results, due to the fact that we are in the linear regime.
For instance the following classical functional spaces (with the associated usual norms) can be considered:
\begin{enumerate}[(a)]
\item \label{XSobolevmp} the Sobolev spaces $W^{m,p}$, with $m\in \mathbb{N}$ and $p\in [1,+\infty]$, 
\item \label{XCm} $C^m([0,L])$, with  $m \in \mathbb{N}$,
\item \label{XBV} $BV(0,L)$.
\end{enumerate}
Then, whatever is $X$ in the above list, $0$ is exponentially stable for \eqref{sys-lin} for the norm $|\cdot|_X$ if and only if there exists $\delta > 0$ such that
\begin{equation}\label{cond-lineaire-1}
\Big\{ \mbox{det} \big(Id_n - \big(\mbox{diag} (e^{- z/\lambda_1}, \cdots,  e^{-  z/\lambda_n})\big)K \big) = 0 , \,z \in \mathbb{C} \Big\}
\implies \Re(z) \le - \delta.
\end{equation}
See, e.g. (the proof of)  \cite[Theorem 3.5 on page 275]{HaleVerduynLunel-book}. For example, for
\begin{equation}\label{specialvalueAK}
n=2, \,
A=
\begin{pmatrix}
1&0
\\
0&2
\end{pmatrix}, \,
K=K_a:=
\begin{pmatrix}
a&a
\\
a&a
\end{pmatrix}, \ \text{with}\  a\in \R,
\end{equation}
condition \eqref{cond-lineaire-1} is equivalent to
\begin{equation}\label{condition-a-1}
  a\in (-1,1/2).
\end{equation}
See \cite[p. 285]{HaleVerduynLunel-book}. Condition \eqref{cond-lineaire-1} is robust to small perturbations on $K$.
However it turns out that this condition is not robust with respect to small perturbations on the $\lambda_i$. 
See again \cite[p. 285]{HaleVerduynLunel-book}. This lack of robustness is of course a problem if one wants to treat the case of a nonlinear system by looking at its linearization at $0$.
We will say that the exponential stability of $0$ is robust with respect to  small perturbations on the $\lambda_i$'s if there exists $\varepsilon>0$ such that, for
every $(\tilde \lambda_1,\tilde \lambda_2,\cdots,\tilde \lambda _n)\in \R^n$ such that
\begin{equation}\label{Lambda}
  |\tilde \lambda_i- \lambda_i|\leq \varepsilon  \quad \mbox{ for } i =1, \cdots, n,
\end{equation}
$0$ is exponentially stable for the perturbed hyperbolic system
\begin{equation}\label{lin-perturbed}
\left\{\begin{array}{ll}
u_t+\tilde \Lambda u=0& \quad \mbox{for}\ (t,x) \in (0, \infty)\times (0,L), \\
u(t,0)=Ku(t,L) & \quad \mbox{for}\ t\in (0,+\infty),
\end{array} \right.
\end{equation}
where
\begin{equation}\label{deftilelambda}
\tilde \Lambda :=\text{diag}({\tilde \lambda_1,\ldots, \tilde \lambda_n}).
\end{equation}
Then, Silkowski (see e.g. \cite[Theorem 6.1 on page 286]{HaleVerduynLunel-book})
proved that for every $X$ in the list \eqref{XSobolevmp}-\eqref{XCm}-\eqref{XBV}, $0$ is exponentially stable  for \eqref{sys-lin} with an exponential stability which is robust
with respect to small perturbations on the $\lambda_i$'s  if and only if
\begin{equation} \label{condition-0}
\rho_0(K)<1,
\end{equation}
where
\begin{equation} \label{defrho0}
\rho_0(K) :  =
\max \Big\{ \rho \big( \mbox{diag} (e^{i \theta_1}, \cdots, e^{ i \theta_n}) K \big); \theta_i \in \R \Big\}<1.
\end{equation}
We point out that condition \eqref{condition-0}, in contrast with condition \eqref{cond-lineaire-1}, does not depend on the
$\lambda_i$'s. For Example \eqref{specialvalueAK}, condition \eqref{condition-0} is equivalent to
\begin{equation}\label{condition-a-2}
  a\in (-1/2,1/2),
\end{equation}
which is more restrictive than condition \eqref{condition-a-1}. \par
\ \par
\noindent
{\it Nonlinear systems.}
Let us now turn to the nonlinear system \eqref{gen-hyp}. The main previous known results on the exponential stability by means of feedback control concern, as far as we know, only the
classical solutions. We introduce some notations. For $p\in [1,+\infty]$, let
\begin{gather}
\label{def|x|p}
\|x \|_p:=\Big(\sum_{i=1}^n|x_i|^p\Big)^{1/p} \quad  \text{for}\ x:=(x_1,\cdots,x_n)^T
\in \R^n,\,   p\in [1,+\infty),
\\
\label{def|x|infty}
\|x \|_\infty:=\max\left\{|x_i|;\, i\in \{1,\cdots,n\}\right\}
\quad  \text{for}\  x:=(x_1,\cdots,x_n)^T
\in \R^n,
\\
\label{def|P|p}
\| M\|_p : = \max_{\|x \|_p = 1} \|M x \|_p\quad
\text{for}\  M \in \R^{n\times n},
\\
\rho_p (M): = \inf \big\{ \|\Delta M \Delta^{-1} \|_p; \; \Delta \in {\cal D}_{n, +} \big\} \quad  \text{for}\  M \in \R^{n\times n},
\end{gather}
where ${\cal D}_{n, + }$ denotes the set of all $n \times n$ real diagonal matrices whose entries on the diagonal are strictly positive.
It is proved in \cite{Coron-Bastin-Andrea-SICON-2008} that
\begin{gather} \label{0<p}
\rho_0(K)\leq \rho_p(K)\qquad \text{for all}\  p\in [1,+\infty], \\
  \label{0=2}
\rho_0(K)=\rho_2(K) \qquad \text{for all}\ n\in \{1,2,3,4,5\},  \\
  \label{0<2}
\text{for all $n>5$, there are $K\in \R^{n\times n}$ such that $\rho_0(K)<\rho_2(K)$. }
\end{gather}
(In fact, concerning \eqref{0<p}, only the case $p=2$ is treated in \cite{Coron-Bastin-Andrea-SICON-2008}; however
the proof given in this paper can be adapted to treat the case of every $p\in [1,+\infty]$.) Let us point out that,
for every $p\in[1,+\infty]\setminus\{2\}$ and every $n\geqslant 2$, there are examples of $K$ such that inequality
\eqref{0<p} is strict.

Now, we give known sufficient conditions  for the exponential stability of $0$ for \eqref{gen-hyp} with respect to
the $|\cdot|_X$-norm.
\begin{enumerate}[(i)]
\item \label{conditionC1} If $X=C^m([0,L])$ with $m\in\mathbb{N}\setminus\{0\}$, a sufficient condition is:
\begin{equation}
\rho_\infty(K)<1.
\end{equation}
This result is proved by  T. H. Qin \cite{1985-Qin-Tie-hu},
Y. C. Zhao \cite{1986-Zhao-Yan-chun}, T. Li \cite[Theorem 1.3 on page 173]{Li-book} and \cite{Coron-Bastin-2014}.
In fact, in \cite{Li-book,1985-Qin-Tie-hu,1986-Zhao-Yan-chun}, $G$ is assumed to have a special structure;
though, it is was pointed out by J. de Halleux et al. in \cite{2003-De-Halleux-et-al-Automatica} that
the case of a general $G$ can be reduced to the case of this special structure. Moreover \cite{Li-book,1985-Qin-Tie-hu,1986-Zhao-Yan-chun,Coron-Bastin-2014} deal only with the case $m=1$;  but the proofs given there can be adapted to treat the general case $m\in\mathbb{N}\setminus\{0\}$.
\item \label{conditionWmp} If $X$ is the Sobolev space $W^{m,p}([0,L])$ with $m\in\mathbb{N}\setminus\{0,1\}$ and $p\in [1,+\infty]$, a sufficient condition is:
\begin{equation}
\rho_p(K)<1.
\end{equation}
The case $p=2$ is treated by J.-M. Coron, B. d'Andr\'{e}a-Novel and G. Bastin in \cite{Coron-Bastin-Andrea-SICON-2008}, and the case of general $p\in[1,+\infty]$ is treated by in J.-M. Coron and H.-M. Nguyen in \cite{Coron-Nguyen-2014}.
In fact \cite{Coron-Bastin-Andrea-SICON-2008,Coron-Nguyen-2014} deal only with the case $m=2$;  nonetheless the proofs given there can be
adapted to treat the general case $m\in\mathbb{N}\setminus\{0,1\}$.
\end{enumerate}
These conditions are only sufficient conditions for exponential stability. It is natural to ask if one can improve them.
In particular, it is natural to ask if the condition $\rho_0(K)<1$ (which seems to be the weakest possible sufficient
condition: see above) is sufficient for the exponential stability in these spaces $X$. 
Of course this is true for $n=1$. However it turns out to be false already for $n=2$: as shown in \cite{Coron-Nguyen-2014}, for every $n\geq2$,
there are analytic maps $A$ such that, for every $\varepsilon>0$, there are $K\in \R^{n\times n}$ satisfying
\begin{equation}
\label{contitionK}
\rho_0(K)=\rho_2(K)<1<\rho_\infty(K)<1+\varepsilon,
\end{equation}
such that, for every $m\in\mathbb{N}\setminus\{0\}$, 0 is not exponentially stable with respect to the $C^m$-norm
for \eqref{gen-hyp} with $G(u):=Ku$. Let us emphasize that, as already mentioned above, the first inequality of \eqref{contitionK} implies that,
for every  $m\in\mathbb{N}\setminus\{0,1\}$, $0$ is exponentially stable with respect to the $W^{m,2}$-norm
for \eqref{gen-hyp}.

Finally, we point out that the right hand side of \eqref{Assumption-On-K-new} is $\rho_1(K)$. Hence
\eqref{Assumption-On-K-new} is equivalent to
\begin{equation} \label{equivrho1<1}
\rho_1(K)<1.
\end{equation}
\begin{remark}
One has
\begin{equation} \label{1=infty}
\rho_1(K)=\rho_\infty(K).
\end{equation}
Indeed, for every matrix $M\in \R^{n\times n}$, one has
\begin{gather} \label{norm1}
  \|M\|_1=\max\left\{\sum_{j=1}^{n}|M_{ij}|;\,i\in\{1,\ldots,n\}\right\},  \\
  \label{norminfty}
  \|M\|_\infty=\max\left\{\sum_{i=1}^{n}|M_{ij}|;\,j\in\{1,\ldots,n\}\right\}.
\end{gather}
In particular $\|M\|_\infty=\|M^T\|_1$, from which one easily gets
\begin{equation} \label{infty-1}
  \rho_\infty(M)=\rho_1(M^T).
\end{equation}
Using \eqref{norm1} and \cite[Lemma 2.4, page 146]{Li-book}, one has
\begin{equation} \label{=spectral}
  \rho_1(M)=\rho_1(|M|)=\rho(|M|),
\end{equation}
where, for $M\in \R^{n\times n}$, $\rho(M)$ is the spectral radius of $M$ and $|M|$ is
the $n\times n$ matrix whose entries are $|M|_{ij}:=|M_{ij}|$.
In particular
\begin{equation} \label{=rho}
  \rho_1(K)=\rho(|K|)=\rho(|K^T|)=\rho_1(K^T),
\end{equation}
which, together with \eqref{infty-1}, implies \eqref{1=infty}.
\end{remark}
\ \par
\noindent
{\bf The context of entropy solutions.} As we mentioned earlier, very few results exist on the stabilization of hyperbolic systems of conservation laws in the context of entropy solutions. In fact, even in the scalar case, we are only aware of the work \cite{Perrollaz-2013}, in which a suitable stationary feedback law is shown to stabilize the solutions exponentially. Regarding the stabilization of hyperbolic systems of conservation laws, we know only two results obtaining asymptotic stabilization in open loop. The first one, obtained by A. Bressan and G.M. Coclite \cite{Bressan-Coclite}, established that for a general hyperbolic system of conservation laws with either genuinely nonlinear or linearly degenerate characteristic fields (in the sense of Lax), and characteristic speeds strictly separated from $0$, one can steer asymptotically in time any initial condition on a finite interval with sufficiently small total variation to all close constant states, by suitably acting on both sides of the interval. Furthermore, the controllability result may depend on the considered class of solutions, as underlined by \cite{Bressan-Coclite}. Indeed, in this example presented in this article, controllability  holds in the context of classical solutions (\cite{LiRao}) but not in the context of weak entropy solutions, emphasizing that linearization techniques cannot be used in the BV class.
The second result, due to F. Ancona and A. Marson \cite{Ancona-Marson-2007} is concerned with a case of a control from a single boundary point rather than on both sides. 

\subsection{Outline}
The article is organized as follows. In Section \ref{Sec-Preliminaries}, we start with some basic remarks that will be needed in the proof of Theorem \ref{Thm-Main}, recalling in particular some embeddings and the solvability of the Riemann problem away from the boundary and on the boundary. Section \ref{Sec-Proof-Main} gives the proof of Theorem \ref{Thm-Main}. It is divided in several steps. First, Section \ref{Sec-Construction} presents the construction of front-tracking approximations of solutions of \eqref{System-U}--\eqref{BoundaryConditions}. Section \ref{Subsec-Glimm} then introduces the Lyapunov functional we will use, which is a suitably weighted Glimm functional, inspired in \cite{Glimm65} and \cite{Coron-Andrea-Bastin-IEEE-2007,Coron-Bastin-Andrea-SICON-2008}. Section \ref{Subsec-Decay-Glimm} proves the exponential decay of this quantity, which is in fact the main step in our analysis. Though, in order to conclude, we shall provide further estimates, in particular guaranteeing that:
\begin{itemize}
	\item our construction is valid for all time (Section \ref{Subsec-T-infinity});
	\item the rarefaction fronts remain small (Section \ref{Subsec-Rarefaction}); 
	\item the solutions are uniformly bounded in $L^\infty(0,L; BV(0,T))$ for all $T>0$  (Section \ref{Subsec-TimeLips}).
\end{itemize}
One can then pass to the limit in our approximate solutions and prove that they converge to a suitable solution of \eqref{System-U}--\eqref{BoundaryConditions} whose $BV$ norm is exponentially decaying, see Section \ref{Subsec-Limit}, thus finishing the proof of Theorem \ref{Thm-Main}.
%
%
%
%
%
%
%
\section{Preliminaries}\label{Sec-Preliminaries}
%
%
%
%
%
%
%
Let us first point out that, using \eqref{Assumption-On-K-new} and replacing, if necessary,  $r_1(u)$ by $\alpha r_1(u)$ and $\ell_1(u)$ by $\alpha^{-1} \ell_1(u)$ for some suitable $\alpha>0$,
we may assume without loss of generality that
\begin{equation} \label{Assumption-On-K}
\max_{k=1,2} \{ |\ell_1(0) \cdot K r_k(0) | + |\ell_2(0)\cdot K r_k(0) |\}< 1.
\end{equation}
For the proof of Theorem~\ref{Thm-Main},
it will be convenient to define, for $u\in BV(0,L)$, the quantity $\mbox{TV}^*_{[0,L]} (u)$ as
\begin{equation}\label{total-var-*}
\mbox{TV}^*_{[0,L]} (u) =  \mbox{TV}_{[0,L]}( u) + |Ku(L-) - u(0+)|,
\end{equation}
where $u(0+)$ and $u(L-)$ have to be understood respectively as the right and left limits of the function $u$ in $x = 0$ and $x = L$, respectively. Note that this quantity is well-defined for all $u \in BV(0,L)$.
\subsection{On the quantity $\mbox{TV}^*_{[0,L]}$}
We prove the following Lemma:
\begin{lemma} \label{Lem-Norm}
Under assumption \eqref{Assumption-On-K}, the quantity $TV^*(0,L)(\cdot)$ is a norm on $BV(0,L)$
which is equivalent to the norm $|\cdot|_{BV}$. Consequently, there exists a constant $C$ such that for all $u \in BV(0,L)$,
\begin{equation} \label{Norm-BV-L-infty}
\norm{u}_{L^\infty(0,L)} \leq C \mbox{TV}^*_{[0,L]}(u).
\end{equation}
\end{lemma}
\begin{proof}
Clearly, there exists a positive constant $C>0$ such that
\begin{equation} \nonumber 
\mbox{TV}^*_{[0,L]}(u) \leq C |u|_{BV}  \qquad \mbox{for all} \ u \in BV(0,L).
\end{equation}
(For this property, of course, \eqref{Assumption-On-K} is not needed.
To get the inequality in the other direction, we notice that for $u \in BV(0,L)$, one has 
\begin{equation} \label{uL-0leq}
|u(L-) - u(0+)| \leq  \mbox{TV}_{[0,L]}(u) \ \text{ and } 	|K u(L-) - u(0+)| \leq  \mbox{TV}^{*}_{[0,L]}(u),
\end{equation}
so
\begin{equation*}
|u(L-) - K u(L-) | \leq 2 \mbox{TV}^{*}_{[0,L]}(u).
\end{equation*}
Expressing $u(L-)$ in the basis $(r_1(0),r_2(0))$:
\begin{equation*}
u(L-) =  a_1 r_1(0) + a_2 r_2(0)
\end{equation*}
and using
$$ K r_k(0) = (\ell_1(0) \cdot Kr_k(0) ) r_1(0) + (\ell_2(0) \cdot Kr_k(0) ) r_2(0), $$
we obtain that $u(L-) - K u(L-)$ has the following coordinates
\begin{equation*}
\begin{pmatrix}
a_1 - a_1 (\ell_1(0) \cdot Kr_1(0) ) + a_2 (\ell_1(0) \cdot Kr_2(0)) \\
a_2 - a_1 (\ell_2(0) \cdot Kr_1(0)) + a_2 (\ell_2(0) \cdot Kr_2(0))
\end{pmatrix}.
\end{equation*}
Due to \eqref{Assumption-On-K}, the matrix
\begin{equation*}
\mbox{Id} - \begin{pmatrix}
\ell_1(0) \cdot Kr_1(0) & \ell_1(0) \cdot Kr_2(0) \\
\ell_2(0) \cdot Kr_1(0) & \ell_2(0) \cdot Kr_2(0)
\end{pmatrix}
\end{equation*}
is invertible and consequently for some positive $C>0$:
\begin{equation*}
|u(L-)| \leq C |u(L-) - K u(L-) |.
\end{equation*}
The conclusion follows easily, using $|u(x) - u(L-)| \leq TV_{[0,L]}(u)$.
\end{proof}
%
%
%
%
%
\subsection{On the Riemann problem}
%
%
%
%
\subsubsection{Usual Riemann problem}
Let us make some brief reminders on the Riemann problem for \eqref{System-U}. Details can be found for instance in \cite[Chapter 5]{Bressan-Book-2000}. \par

Following \cite[Section 5.2]{Bressan-Book-2000}, we introduce the Lax curves as the curve obtained by gluing of the admissible part of the Hugoniot locus and of the rarefaction curves, i.e.
\begin{equation} \label{Lax-Curve}
	\Psi_k(\sigma, u) =
		\left\{
			\begin{array}{ll}
				S_k(\sigma, u) \quad &\hbox{ if }€ \sigma <0,
				\\
				R_k(\sigma, u) \quad &\hbox{ if } \sigma \geq 0.
			\end{array}
		\right.			
\end{equation}
Here, $R_k(\sigma,u)$ corresponds to the rarefaction curves, that is, the orbits of the vector fields $r_{k}$:
$$ \frac{dR_k(s,u)}{ds} = r_k(R_k(s,u)), \quad s \in [0,\sigma], \qquad R_k(0,u) = u. $$
The part corresponding to $\sigma \geq 0$ (due to \eqref{Genuine-Non-Linear}) is composed of points $u_{+}$ which can be connected to $u$ from left to right by a rarefaction wave:
\begin{equation} \label{Rarefaction}
u(t,x)= \left\{ \begin{array}{ll}
u & \text{ if } x  <  \lambda_{k}(u) t, \\
R_{i}(\sigma,u) & \text{ if } x  = \lambda_{k}(R_{k}(\sigma,u)) t, \\
u_{+} & \text{ if } x > \lambda_{k}(u_{+}) t.
\end{array} \right.
\end{equation}
On the other side $S_k(\sigma, u)$ stands for the shock curve, which describes the $k$-th branch of the Hugoniot locus which gathers point $u_{+}$ satisfying for a fixed state $u$
the Rankine-Hugoniot condition:
\begin{equation} \label{Eq:RankineHugoniot}
f({u}_{+}) - f({u}) = {s} \, \big( u_{+} - u \big), \ \ s \in \R.
\end{equation}
To be slightly more precise, the $k$-th shock curve is defined for $\sigma$ small as follows:
$$ u_+ = S_k(\sigma,u) $$
if and only if $u_+$ and $u$ satisfy
$$ u_+ - u = \sigma r_k(u_+,u), $$
where $r_k(u_+, u)$ is the (suitably normalized) $k$-th eigenvector of the matrix
$$ A(u_+,u) = \int_0^1 Df( u + t(u_+ - u)) \, dt. $$
For $\sigma <0$, points $u_{+}=S_{k}(\sigma,u)$ can be connected to $u$ from left to right by an admissible shock wave:
\begin{equation} \label{Choc}
u(t,x)= \left\{ \begin{array}{ll}
u & \text{ if } x < s t, \\
u_{+} & \text{ if } x > s t,
\end{array} \right.
\end{equation}
where the shock speed $s$ is equal to $\lambda_k(u_+,u)$, the $k$-th eigenvalue of the above matrix $A(u_+,u)$. \par
We recall that, under suitable parameterization, the function $(\sigma, u) \mapsto \Psi_k(\sigma, u)$ is of class $C^2$, see for example \cite[p. 99]{Bressan-Book-2000}. \par
Finally, following Lax \cite{Lax}, using the implicit function theorem, we see that there exists $\delta >0$ such that for all $u_-$ and $u_+$ with $|u_-| \leq \delta$ and $|u_+| \leq \delta$, there exists $\sigma_1$ and $\sigma_2$ small such that
\begin{equation} \label{ConnectingU-U+}
u_+ = \Psi_2(\sigma_2, \Psi_1(\sigma_1, u_-)),
\end{equation}
and besides, there exists a constant $C>0$ such that
$$ \frac{1}{C} |u_+ - u_-| \leq |\sigma_1| + |\sigma_2 | \leq C |u_+ - u_-|. $$
See for instance \cite[Theorem 5.3]{Bressan-Book-2000}. \par
In the following, when considering two states $u_-$ and $u_+$ satisfying $u_+ = \Psi_k(\sigma,u_-)$, we will say that $u_-$ and $u_+$ are connected through a $k$-wave of strength $|\sigma|$. If $\sigma \geq 0$, this wave is a $k$-rarefaction, while it is a $k$-shock otherwise. \par
\subsubsection{Boundary Riemann problem}
\label{SubsecBRP}
Since we are considering solutions to an initial-boundary value problem, we also need to consider the Riemann problem on the boundaries. Due to the assumption \eqref{PositiveVel}, we only consider the problem on the left side of the interval. Actually, under this assumption, this problem is fairly simple and solve as in the usual case. Let us consider indeed on the boundary $x=0$ the conditions:
\begin{equation*}
u(t,0) = \left\{ \begin{array}{l}
u_{-} \text{ for } t >t_{0} , \\
u_{+} \text{ for } t <t_{0} ,
\end{array} \right.
\end{equation*}
with $u_{-}$ and $u_{+}$ sufficiently small. Then writing again \eqref{ConnectingU-U+}, this problem can be solved as in the usual case by a $1$-wave followed by a $2$-wave (from left to right, that is from top to bottom). \par
Note that under the constraint \eqref{BoundaryConditions}, boundary conditions $u_{0}$ at $x=0^{+}$ and $u_{L}$ at $x=L^—$ generate a Riemann problem between $u_{0}$ and $K u_{L}$ at $x=0$.
%
%
%
%
%
%
%
%
%
\section{Proof of Theorem \ref{Thm-Main}}\label{Sec-Proof-Main}
In the following, we work under the assumptions of Theorem \ref{Thm-Main}, and in particular \eqref{PositiveVel} and \eqref{Assumption-On-K}.
Theorem \ref{Thm-Main} is based on a construction of solutions of \eqref{System-U}--\eqref{BoundaryConditions} relying on a wave-front tracking algorithm. This algorithm generates approximations of a solution $u$ of \eqref{System-U}, which have a particular shape.
To be more precise, for $h$ small, we look for $u_{h} = u_{h}(t,x)$ defined for $t\geq 0$ and $x \in [0,L]$ such that:
\begin{itemize}
	\item $u_{h}$ is a piecewise constant function on $\R_{+} \times [0,L]$, with a finite number of discontinuities (locally in time), which are straight lines (called {\it fronts}),
	\item each front is either a {\it rarefaction front} or a {\it shock}. In the former case, the states on the sides of the discontinuities are connected for a rarefaction wave, in the latter by a shock,
	\item the rarefaction fronts are of strength ${\mathcal O}(h)$,
	\item the boundary condition is satisfied for all times (taking left and right limits at discontinuity points),
	\item a quantity equivalent to the $TV^*$ norm of $u(t)$ decays exponentially as time evolves.
\end{itemize}
%
%
%
%
%
\subsection{Construction of front-tracking approximations}
\label{Sec-Construction}
We mainly follow R. DiPerna's strategy \cite{DiPerna-1976} consisting, starting from a piecewise constant approximation of the initial data, in solving the generated Riemann problems and replacing rarefaction waves with piecewise constant approximations called {\it rarefaction fans}. When two discontinuities meet, the process is iterated, with an important convention on rarefactions, that is {\it they are not re-split} across an interaction. See below for a more precise description. This means in particular that there are two ways to treat rarefaction waves in the process; let us describe these two methods. \par
\ \par
\noindent
{\bf Rarefaction fans.} A $k$-rarefaction wave $(u_{-},u_{+})$ with $u_{+}=R_{k}(\sigma,u_{-})$, $\sigma>0$, centered at $(\overline{t},\overline{x})$ can be approximated by a $k$-rarefaction fan (of accuracy $h$) as follows. Set
$$ p = \lceil \sigma/ h \rceil. $$
If $p >1$, i.e. if $\sigma >h$, we define the intermediate states for $j \in \{1, \cdots, p\}$:
$$ u_{j} = R_k(\sigma/p, u_{j-1}), \quad u_{0} = u_{-}, $$
and define
$$ x_{j+1/2} (t) = \overline{x} + \lambda_k(u_{j+1}) (t-\overline{t}). $$
In that case, the rarefaction wave is approximated locally by
\begin{equation} \label{U-h-Fan-front}
\tilde{u}(t,x) = \left\{ 
\begin{array}{ll}
	u_{0} = u_{-} \quad & \hbox{for } x < x_{1/2} (t) \\
	u_{j} \quad & \hbox{for } x \in (x_{j-1/2}(t), x_{j+1/2}(t)),\, j \in \{1,\cdots,p_k-1\}, \\
	u_{p} = u_{+} \quad & \hbox{for } x > x_{p-1/2}(t).
\end{array} \right.
\end{equation}
If $p <1$, i.e. if $\sigma \in (0,h]$, we simply set
\begin{equation*}
\tilde{u}(t,x) = \left\{ \begin{array}{l}
				u_{-} \quad \hbox{ for } x < x_{1/2} (t) \\
				u_{+} \quad \hbox{ for } x > x_{1/2} (t),
			\end{array} \right.
\end{equation*}
\ \par
\noindent
\noindent {\bf Approximating rarefaction waves by a single front.} In that case, the above Riemann problem is solved using only one front, i.e. the approximate solution $\tilde{u}$ is locally given by
\begin{equation} \label{U-h-Rarefaction-1front}
\tilde{u}(t,x) = \left\{ \begin{array}{l}
				u_{-} \quad \hbox{ for } x < \overline{x} + \lambda_k(u_{+}) (t-\overline{t}), \\
				u_{+} \quad \hbox{ for } x > \overline{x} +  \lambda_k(u_{+}) (t - \overline{t}).
			\end{array} \right.
\end{equation}
\ \par
\noindent
\noindent {\bf The wave-front tracking algorithm.} We start from an approximate sequence $u_{0,h} \in BV(0,L)$, $h>0$, of the initial condition, satisfying
\begin{equation}
	\label{Conv-u-0-h}
\| u_{0,h} \|_{\infty} \leq \| u_{0} \|_{\infty}, \ \ TV(u_{0,h}) \leq TV(u_{0}), \ \ u_{0,h} \underset{h\to0}\longrightarrow u_{0} \text{ in } L^{1}(0,L).
\end{equation}
Now the construction of an approximate solution of \eqref{System-U} is then done as follows: \par
\ \par
\noindent 1. At time $t = 0$, we construct $u_h$ as the solution of the Riemann problems for $u(t= 0)$ for which all rarefaction waves are replaced by rarefaction fans (with accuracy $h$). This includes the Riemann problem generated by $u_{0,h}(0^{+})$ and $K u_{0,h}(L^{-})$. We extend the resulting discontinuities (called fronts) as straight lines, until two of them meet (at a point called an {\it interaction point}), or until one of them meets the boundary (which is necessary the right one, since all fronts have positive speeds under the assumption \eqref{PositiveVel}). \par
\ \par
\noindent 2. When a front hits the (right) boundary, we solve the corresponding Riemann boundary problem between $K u_h(t-,L{-})$ and $u_h(t-,0{+})$. Again in that case we approximate all outgoing rarefaction waves by rarefaction fans. \par
\ \par
\noindent 3. When two fronts (say of family $k$ and $\ell$) interact at some time $t$ in some point $x \in (0,L)$:
\begin{itemize}
\item if $k = \ell$, solve the resulting Riemann problem between the leftmost and rightmost states and approximate the outgoing rarefaction wave of the family $k$ (if any) by a single front and the outgoing rarefaction wave of the other family (if any) by a rarefaction fan.
\item if $k \neq \ell$, solve the Riemann problem and approximate each outgoing rarefaction wave (if any) by a single front.
\end{itemize}
\ \par
\noindent 4. If at some time, three fronts (or more) interact in the interior of the domain or two fronts (or more) interact at the boundary $ x = L$, we slightly change the velocity of one of the incoming fronts so that there is only two fronts meeting at the same time. These changes of velocity are done so that the new (constant) velocity $c$ belongs to an $h$-neighborhood of the expected velocity. We also modify the velocities similarly to avoid having several interactions at the same time. (This is classical in the context of front tracking approximation, see \cite{Bressan-Book-2000}.) \par
\ \\
Our construction works as long as the number of interactions of fronts is finite and $u_h$ stays in the set where we can solve the Riemann problem. Let us therefore define
\begin{multline} \label{DefT*}
T_h^* = \sup\{ t >0, \hbox{ such that the number of fronts is finite in } (0,t) \times (0,L) \\
		\hbox{ and } \norm{u_h}_{L^\infty((0,t)\times (0,L))} \leq \delta_0\}, 
\end{multline}
where $\delta_0 \in (0, \delta)$ is a positive parameter fixed below by \eqref{PositiveVelocities-c}--\eqref{Feedback-Condition}, $\delta$ being the parameter in \eqref{ConnectingU-U+}.
\begin{remark}
Note that $T_h^*>0$ as the fronts propagate at bounded velocities and since the initial data is piecewise constant. We will later show that $T^*_h$ actually is infinite for all $h>0$, see Lemma \ref{Lemma-T-Infinity}.
\end{remark}
%
%
%
%
%
%
\subsection{A Glimm-type functional}\label{Subsec-Glimm} 
In order to get estimates on this approximate solution $u_h$ of \eqref{System-U}--\eqref{BoundaryConditions}, we will introduce a functional resembling Glimm's one \cite{Glimm65} and adapted to our problem. \par
We choose $\delta_0 \in (0, \delta)$, $c_* >0$, $\gamma >0$, $\varepsilon >0$ such that
\begin{align}
\label{PositiveVelocities-c} 
& c_* < \min_{ |u| \leq \delta_0} \lambda_1(u) < \max_{ |u| \leq \delta} \lambda_1(u) < \min_{ |u| \leq \delta} \lambda_2(u),\\
\label{Feedback-Condition}
& \max_{ |u| \leq \delta_0} \max_{k \in \{1, 2\}} \{ |\ell_1(Ku) \cdot Kr_k ( u) | + |\ell_2(Ku) \cdot Kr_k(u)  |\}  <  \exp(-\gamma L) - \varepsilon.
\end{align}
This can be done according to the assumptions \eqref{PositiveVel} and \eqref{Assumption-On-K}.

Next, for a piecewise constant function $U$ on $[0,L]$ with $TV^*(U)$ small enough:
\begin{itemize}
\item we denote $x_1 < x_2 <\cdots <x_n$ the set of discontinuities in $(0,L)$,
\item for each $i \in \{1, \cdots, n\}$, we let $u_{i,-}$ and $u_{i,+}$ the limits of $U$ at $x_i$ from the left and from the right respectively,
\item for each $i \in \{1, \cdots, n\}$, we introduce the values $\sigma_{i,1}$ and $\sigma_{i,2}$ such that
\begin{equation*}
u_{i,+} = \Psi_2(\sigma_{i,2}, \Psi_1(\sigma_{i,1}, u_{i,-})), 	
\end{equation*}	
which are obtained by the solvability of the Riemann problem,
\item finally, we quantify the ``interaction'' on the boundary between $K u(L-)$ and $u(0+)$ by $\sigma_{0,1}$, $\sigma_{0,2}$ such that
\begin{equation*}
u(0+) = \Psi_2(\sigma_{0,2}, \Psi_1(\sigma_{0,1}, Ku(L-))), 	
\end{equation*}	
whose existence is granted by the solvability of the boundary Riemann problem.
\end{itemize}
Now for such a piecewise constant function $U$, we introduce the following functionals:
\begin{eqnarray}
\label{Linear-Functional}
V(U) & = & \sum_{i=0}^n \left( |\sigma_{i,1}| + |\sigma_{i,2}| \right) e^{-\gamma x_i },  \\
\label{Quadratic-Functional}
Q(U) & = &  \sum_{(x_i, \sigma_i)}  |\sigma_i| e^{-\gamma x_i } \Bigg(\sum_{(x_j,\sigma_j) \hbox{\scriptsize \, approaching }(x_i, \sigma_i)} |\sigma_{j}|    e^{-\gamma x_j }\Bigg),
\end{eqnarray}
where a front $(x_j,\sigma_j)$ of a family $k_j$ is said to be approaching of a front $(x_i, \sigma_i)$ of a family $k_i$ if and only if one of the conditions is satisfied:
\begin{itemize}
	\item $x_j < x_i$ and $(k_i,k_j) = (1,2)$,
	\item $x_j < x_i$, $k_i = k_j$ and at least one of $\sigma_j$ or $\sigma_i$ is negative.
\end{itemize}
We underline that the difference with the usual functionals of total strength and of interaction potential lies in the exponentials, and on the fact that the boundary is taken into account at index $i=0$, emphasizing the special role played by the feedback operator. \par
As seen from the construction, at all times except in a discrete set, namely the times of interaction of fronts, the discontinuities in the approximation $u_h$ given by the front tracking algorithm are connected through a wave of family $1$ or $2$, that is, for almost all $t>0$, for each $x_{i}(t)$, either $\sigma_{i,1}(t)$ or $\sigma_{i,2}(t)$ vanishes. In that case, we will say that $x_i(t)$ corresponds to a $k$-front, where $k \in \{1,2\}$ is such that $\sigma_{i,k} \neq 0$. \par
For later use, let us also point out that two rarefaction fronts of the same family cannot meet. \par
\ \par
\noindent
{\bf Notations.} For sake of simplicity, we will use the slight abuse of notation $V(t) = V(u_h(t))$ and $Q(t) = Q(u_h(t))$. \par
\subsection{Decay of Glimm's functional}\label{Subsec-Decay-Glimm}

Our goal is to prove the following result:

\begin{lemma}\label{Lem-Lyapunov}
Under the setting of Theorem \ref{Thm-Main}, there exist positive constants $\varepsilon_0>0$, $c_0>0$, $\nu>0$ such that for any $h>0$, if we define the approximate solution $u_h(t)$ of \eqref{System-U} using the front tracking approximation explained above, starting from $u_{0,h}$ satisfying
\begin{equation} \label{Init-V-small}		
V(u_{0,h}) \leq \varepsilon_0, 
\end{equation}
then the functional
\begin{equation} \label{Def-J}
J(t) = V(u_h(t))+c_0 Q (u_h(t))
\end{equation}
satisfies
\begin{equation} \label{Exp-Decay-J}
J(t) \leq \exp(- \nu t) J(0), \qquad t \in [0, T_h^*).
\end{equation}
\end{lemma}
We recall that the approximation is well-defined till $T^{*}$, see \eqref{DefT*}. \par

\begin{proof}
We split the proof in two parts: first we study the evolution of $V$ and $Q$ in various cases, and then we come back to the functional $J$. \par
\ \\
\noindent
{\bf 1. Behavior of $V$ and $Q$.} \par
We fix a time $t_0 \in (0,T^*_h)$. We recall that interactions in the domain only involve two fronts and interaction on the boundary only involve one. We discuss the evolution of the functionals $V$ and $Q$ locally around $t_0$ according to three cases (see Figures \ref{fig-case-1-2} and \ref{fig-case-3}):
\begin{itemize}
\item \textbf{Case 1:} There is no interaction at $t=t_0$ in $(0,L)$, nor on the boundary.
\item \textbf{Case 2:} There is a front interaction at $t = t_0$ in $(0,L)$.
\item \textbf{Case 3:} There is a front hitting the boundary at $t = t_0$.
\end{itemize}
\begin{figure}[h]
\begin{center}
\begin{tikzpicture}[scale=0.59]
    \draw [thick] (0,1) -- (8,1);
    \draw [thick] (0,1) -- (0,10);
    \draw [thick] (8,1) -- (8,10);
    \draw (0,6) -- (8,6);

    \draw [thick] (12,1) -- (20,1);
    \draw [thick] (12,1) -- (12,10);
    \draw [thick] (20,1) -- (20,10);
    \draw (12,6) -- (20,6);

    \node [left] at (0,6) {$t=t_0$};
    \node [right] at (20,6) {$t=t_0$};

    \draw [blue] (1,4) -- (2,9);
    \draw [blue] (3,4) -- (5,8);
    \draw [blue] (5,4) -- (6,8);

    \draw [blue] (14,4) -- (16,8);
    \draw [blue] (14,2) -- (16,10);
    \draw [blue] (17,4) -- (18,7);

    \node at (4,0) {Case 1};
    \node at (16,0) {Case 2};
\end{tikzpicture}
\end{center}
\caption{Left, Case 1: no interaction at time $t_0$. Right, Case 2: two fronts interact at $t = t_0$.}
\label{fig-case-1-2}
\end{figure}
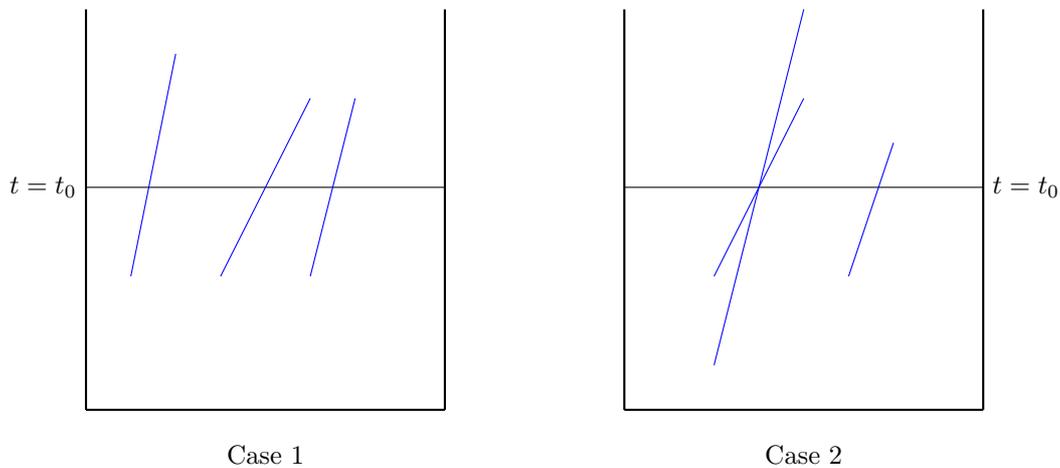

\noindent \textbf{Case 1.} We suppose that there is no interaction at $t=t_0$. Since there is a finite number of fronts at time $t_0$, there exists a neighborhood $\mathscr{T}$ of $t_0$ on which there is no interaction. \par
As the approximation $u_h(t_0)$ is a piecewise constant function, we call as before $x_1(t) < x_2(t)<\cdots < x_{n}(t)$ the discontinuities at time $t \in \mathscr{T}$ and $u_{i,-},\,  u_{i,+}$ the limits of $u_h(t,\cdot)$ at $x_i$ from the left and from the right. Note that by construction, each trajectory $t\mapsto x_i(t)$ corresponds to a $k_i$-front and the states $u_{i,-}$ and $u_{i,+}$ on its sides do not depend on $t \in \mathscr{T}$. Therefore, each front has constant strength and velocity and moreover, due to condition \eqref{PositiveVelocities-c}, one has $x_i'(t) \geq c_*$. Therefore, we obtain, for all $t \in \mathscr{T}$,
\begin{equation} \label{V-Q-Case1}
 \frac{dV}{dt}(t)\leq -c_* \gamma  V(t), \qquad 	 \frac{dQ}{dt}(t)\leq -2c_* \gamma Q(t).
\end{equation}
\noindent
\textbf{Case 2.} We suppose that two fronts interact in $(0,L)$. Here it necessary to relate the strength of the outgoing fronts with the incoming ones. Following \cite{Glimm65} (see also \cite[Lemma 7.2]{Bressan-Book-2000}), we have the following Glimm's interaction estimates:
\begin{lemma}[\cite{Glimm65}]
\label{Lemma-Interaction}
Consider an interaction between two incoming wave-fronts and assume that the left, middle and right states $u_L, u_M$ and $u_{R}$ satisfy $|u_L| \leq \delta$, $|u_M| \leq \delta$ and $|u_R| \leq \delta$.
\begin{itemize}
\item \emph{Distinct family:} Let $\hat \sigma_1$ and $\hat \sigma_2$ be the sizes of two incoming fronts corresponding respectively to the family $1$ and $2$. Then the outgoing fronts have strength $\sigma_1, \, \sigma_2$ satisfying:
\begin{equation} \label{Interacting-Fronts-1}
|\sigma_1 - \hat \sigma_1| + |\sigma_2 - \hat \sigma_2| \leq C_{\delta} |\hat \sigma_1|\, |\hat \sigma_2|.
\end{equation}
\item \emph{Same family:} Let $\tilde \sigma, \, \hat \sigma$ be the sizes of two incoming fronts both belonging to the $k$-th characteristic family. Then the outgoing fronts have strength $\sigma_k$ for the outgoing fronts belonging to the $k$-th characteristic family and $\sigma_{k'}$ for the other family and satisfy:
\begin{equation} \label{Interacting-Fronts-2}
|\sigma_k - (\tilde \sigma+ \hat \sigma)| + |\sigma_{k'}| \leq C_{\delta} |\tilde \sigma|\, |\hat \sigma| \left(|\tilde \sigma|+  |\hat \sigma|   \right).
\end{equation}
\end{itemize}
\end{lemma}
We assume that the interaction takes place at $t = t_0$, $x = x_i$, and we denote by $\sigma_{i}(t_0-)$, $\sigma_{i+1}(t_0-)$ the strength of the incoming waves, and by $\sigma_{i,1}(t_0+)$, $\sigma_{i,2}(t_0+)$ the strength of the outgoing waves of families $1$ and $2$. \par
Using Lemma \ref{Lemma-Interaction} and arguing as in \cite{Glimm65} (see also \cite{Bressan-Book-2000,DiPerna-1976}), we then obtain
\begin{equation} \nonumber
	 V(t_0 +)-V(t_0 -) \leq C_\delta |\sigma_{i}(t_0-)| |\sigma_{i+1}(t_0-)| e^{-\gamma x_i} \leq C_\delta |\sigma_{i}(t_0-)| |\sigma_{i+1}(t_0-)| ,
\end{equation}
while
\begin{equation} \nonumber
	Q(t_0+) - Q(t_0-)\leq  |\sigma_{i}(t_0-)| |\sigma_{i+1}(t_0-)| \left(-e^{-2 \gamma x_i}  + e^{-\gamma x_i} C_\delta V(t_0-) \right),
\end{equation}
where $V(t_0-), V(t_0+), Q(t_0-), Q(t_0+)$ denote the respective limits of $V(t_0-\tau)$, $V(t_0 + \tau)$, $Q(t_0- \tau)$ and $ Q(t_0 + \tau)$ as $\tau$ goes to $0^{+}$. \par
In particular, provided that
\begin{equation} \label{Condition-V-t-0Case2}
2 C_\delta e^{2 \gamma L} V(t_0-) \leq 1,
\end{equation}
we have
\begin{equation} \label{Decay-Q-inside}
Q(t_0+) - Q(t_0-)\leq  - \frac{1}{2} e^{-2 \gamma L} |\sigma_{i}(t_0-)| |\sigma_{i+1}(t_0-)|,
\end{equation}
and
\begin{equation}
	\label{Decay-Case2}
V(t_0+) + c_0 Q(t_0+) \leq V(t_0-) + c_0 Q(t_0-),
\end{equation}
for the choice
\begin{equation} \nonumber
c_0 = 2 C_\delta e^{2 \gamma L}.
\end{equation}
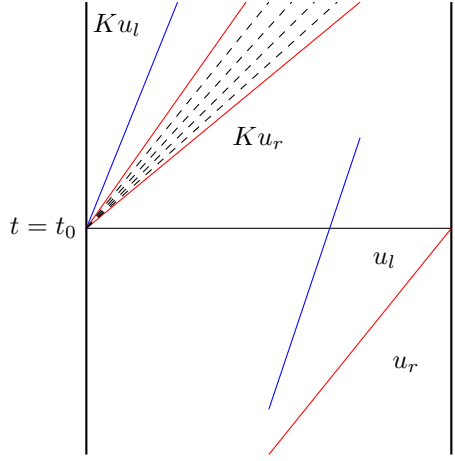
\begin{figure}[ht]  \label{fig-case-3}
\begin{center}
\begin{tikzpicture}[scale = 0.6]
    \draw [thick] (0,1) -- (0,11);
    \draw [thick] (8,1) -- (8,11);
    \draw (0,6) -- (8,6);

    \node [left] at (0,6) {$t=t_0$};

    \draw [red] (4,1) -- (8,6);
    \draw [blue] (0,6) -- (2,11);
    \draw [blue] (4,2) -- (6,8);
    \draw [red] (0,6) -- (6,11);

    \draw [red] (0,6) -- (3.5,11);
    \draw [dashed] (0,6) -- (4,11);
    \draw [dashed] (0,6) -- (4.5,11);
    \draw [dashed] (0,6) -- (5,11);
    \draw [dashed] (0,6) -- (5.5,11);

    \node [left] at (7,5.3) {$u_l$};
    \node [right] at (6.5,3) {$u_r$};
    \node [left] at (1.4,10.5) {$K u_l$};
    \node [right] at (3,8) {$K u_r$};

\end{tikzpicture}
\end{center}
\caption{Case 3: A front hitting the boundary at $t=t_0$.}
\end{figure}
\noindent\textbf{Case 3.} Here we suppose that one front hits the boundary at time $t = t_0$, $x = L$. The following lemma helps us describe the fronts created in $x= 0$.
\begin{lemma} \label{Lemma-Boundary}
Assume that a $k$ front of strength $\hat \sigma$ hits the boundary at $t = t_0-$, i.e.
\begin{equation} \label{Cond-u-r-u-l}
u_{R} = \Psi_k(\hat \sigma, u_{L}).
\end{equation}
Then it generates at $t = t_0+$ and $x = 0$ two waves at $x = 0$ so that
\begin{equation} \label{Def-sigma-boundary}
K u_{R} =  \Psi_2(\sigma_{2}, \Psi_1(\sigma_{1}, K u_{L})),
\end{equation}
with
\begin{equation} \label{Est-sigma-1-2-boundary}
|\sigma_1 - \hat \sigma  \ell_1(K u_L) \cdot K r_k(u_L) | + |\sigma_2 -\hat \sigma  \ell_2(K u_L) \cdot K r_k(u_L) | \leq C_\delta |\hat \sigma|^2.
\end{equation}
\end{lemma}
\begin{proof}
Since the front tracking solution  satisfies the boundary condition so at $t=t_0+$, it will generate the following Riemann problem from the left boundary
\begin{equation*}
u(t,0) = \left\{ \begin{array}{llll}
K u_L\  \mbox{for} \ t=t_0+,  \\
K u_R \ \mbox{for} \ t=t_0+.
\end{array} \right.
 \end{equation*}
Lemma \ref{Lemma-Boundary} amounts to get estimates on $\sigma_1$, $\sigma_2$ in \eqref{Def-sigma-boundary} from $\hat \sigma$ in \eqref{Cond-u-r-u-l}. \par
In order to do that, we consider the map
\begin{equation*}
g(\sigma, \sigma_1, \sigma_2) = K \Psi_k(\sigma, u_L) - \Psi_2( \sigma_2, \Psi_1 (\sigma_1, K u_L)),
\end{equation*}
defined for $\sigma, \sigma_1, \sigma_2$ in a neighborhood of $0$. Thanks to the $C^2$ regularity of Lax's curves, this map is $C^2$. Besides,
\begin{equation*}
\partial_{\sigma_1} g(0,0, 0)  = -r_1(K u_L),	\quad	\partial_{\sigma_2} g(0,0, 0)  = -r_2(K u_L).
\end{equation*}
As the eigenvectors $r_1(Ku_L)$ and $r_2(Ku_L)$ are linearly independent, by the implicit function theorem there exists a $C^2$ function $\Phi$ such that in a neighborhood of $0$,
\begin{equation*}
g(\sigma, \sigma_1, \sigma_2) = 0 \ \hbox{ if and only if } \ (\sigma_1, \sigma_2) = \Phi(\sigma).
\end{equation*}
Besides, close to $\sigma = 0$, we have the estimate
\begin{equation*}
| \sigma K r_k(u_L) - \sigma \Phi_2'(0) r_2(K u_L) -\sigma \Phi_1'(0) r_1(K u_L)| \leq C_\delta |\sigma|^2, 
\end{equation*}
so that we immediately get
\begin{equation*}
\Phi_1'(0) = \ell_1(K u_L)  \cdot K r_k (u_L) , \quad \Phi_2'(0) =   \ell_2(K u_L) \cdot K r_k (u_L) .
\end{equation*}
As $\Phi(0) = 0$, we deduce \eqref{Est-sigma-1-2-boundary} for $\hat \sigma$ small enough.
\end{proof}	
Now according to \eqref{Feedback-Condition}, in the situation of Lemma \ref{Lemma-Boundary}, we obtain
\begin{equation} \label{interaction1}
|\sigma_1| + |\sigma_2| \leq  |\hat \sigma| (e^{-\gamma L} - \varepsilon),
\end{equation}
so that
\begin{equation} \nonumber
V(t_0+) - V(t_0-) \leq - \varepsilon |\hat \sigma|.
\end{equation}
Besides, one easily gets that
\begin{equation*}
Q(t_0+ ) - Q(t_0- ) \leq V(t_0-) (|\sigma_1| + |\sigma_2|) \leq V(t_0-) |\hat \sigma|.
\end{equation*}
It follows that
\begin{equation}
	\label{Estimate-Case3}
(V(t_0+) + c_0 Q(t_0+)) - (V(t_0-) + c_0 Q(t_0-)) \leq |\hat \sigma| (- \varepsilon + c_0 V(t_0-)).
\end{equation}
\noindent
{\bf 2. Decay of the functional $\mathbf{J(t) = V(t) + c_0 Q(t)}$.} \par
Let us now prove that if we assume
\begin{equation} \label{Init-Smallness}
J(0)  \leq \min\left\{ \varepsilon, \frac{1}{2C_\delta e^{2 \gamma L} } \right\},
\end{equation}
then the functional $J$ is exponentially decaying on $[0,T_h^*)$. \par
Indeed, suppose that \eqref{Init-Smallness} satisfied. We denote by $t_1 = 0$ and $\{t_k\}_{k=1}^\infty$ the increasing sequence of times such that at $t=t_k (k\geq 2)$, two fronts interact or some fronts hits the right boundary. We have proved the following:
\begin{itemize}
	\item from Case 1 and estimate \eqref{V-Q-Case1}:  for all $t \in (t_k, t_{k+1})$, one has $J(t) \leq  e^{-c_* \gamma (t- t_k)} J(t_k)$;
	\item if two fronts interact at some time $t_k$ corresponding to Case 2, then from estimates \eqref{Decay-Case2}, $J(t_k+) \leq J(t_k-)$, provided \eqref{Condition-V-t-0Case2}, i.e. $2C_\delta e^{2 \gamma L}V(t_k-)\leq 1$;
	\item if one front hits the boundary, from Case 3 and estimate \eqref{Estimate-Case3}, then $J(t_k+) \leq J(t_k-)$ provided $V(t_k-) \leq \varepsilon$.
\end{itemize}
Therefore, if $J(0)$ satisfies \eqref{Init-Smallness}, which can be ensured by choosing $V(0)$ small enough since
\begin{equation*}
J(0) \leq V(0) + c_0 V(0)^2,
\end{equation*}
then $J$ decreases and $V = V(t)$ always satisfies
\begin{equation*}
V(t) \leq \min\left\{ \varepsilon, \frac{1}{2C_\delta e^{2 \gamma L}}\right\}.
\end{equation*}
Combining the above estimates, we thus obtain
\begin{equation} \label{Exp-Decay-J-t-*}		
J(t) \leq e^{-c_* \gamma t} J(0), \quad \hbox{ for all } t \in (0,T_h^*).
\end{equation}
This concludes the proof of Lemma \ref{Lem-Lyapunov}.
\end{proof}
%
%
%
%
%
%
%
\subsection{Global definiteness of the approximations}\label{Subsec-T-infinity}
In this subsection, we prove the following.
\begin{lemma}
	\label{Lemma-T-Infinity}
	There exists $\varepsilon_1 \in (0, \varepsilon_0]$, where $\varepsilon_0$ is the constant introduced in Lemma \ref{Lem-Lyapunov}, such that for small $h>0$, the front-tracking approximation $u_{h}$ starting from an initial data $u_{0,h}$ satisfying 
	\begin{equation} \label{Init-V-small-1}		
	V(u_{0,h}) \leq \varepsilon_1, 
	\end{equation}
	is defined globally in time, that is, $T^*_h = \infty$.
\end{lemma}
\begin{proof}
We argue similarly as in \cite[Section 7.3, item 4]{Bressan-Book-2000}. \par
First, we notice that $TV^*(\cdot)$ in \eqref{total-var-*} and $V$ are equivalent in the sense that for come $C>0$,
\begin{equation} \label{VTV*Equiv}
C^{-1} \, V(U) \leq TV^{*}(U) \leq C \, V(U),
\end{equation}
for all piecewise constant functions $U$ on $[0,L]$ with suitably small $BV$ norm. From the estimate \eqref{Norm-BV-L-infty}, it is clear that if $u_{0,h}$ satisfies \eqref{Init-V-small-1} for some $\varepsilon_1 \in (0, \varepsilon_0]$ small enough, $\norm{u_{0,h}}_{L^\infty(0,L)} \leq \delta_0$, so that $T_h^*>0$. Then, using the decay estimates \eqref{Exp-Decay-J} and estimates \eqref{Norm-BV-L-infty} and \eqref{VTV*Equiv}, taking $\varepsilon_1 \in (0 , \varepsilon_0]$ smaller if necessary, $J(0)$ can be made small enough to obtain $\norm{u_h(t)}_{L^\infty(0,L)} \leq \delta_0/2$ for all $t \in (0,T_h^*)$. \par
Therefore, the only reason for $T_h^*$ to be finite is the possible accumulation of interaction times. Let us show that this cannot happen. 
Set 
\begin{equation} \label{DefM}
M > \max_{|u|\leq \delta} \lambda_2(u),
\end{equation}
and $T_M = L /M$, and define the domain
\begin{equation*}
D(t_0) = \{ (t,x) \in [0, t_0 + T_M] \times [0,L], \hbox{ with }  M (t-t_0) \leq x \}.
\end{equation*}

Let us first show that in $D(0)$, there can be only finitely many occurrences of interactions from which more than two fronts exit. 
Indeed, the only fronts interacting in $D(0)$ are the ones coming from $t = 0$, $x \in (0,L)$ and their descendants. By that, we mean that tracing a front from this zone backward in time across interactions, we end up always at a point in $(0,L)$ at $t=0$, not at the boundary. \par
Let us then consider the following functionals defined on piecewise constant functions $U$ with suitably small $BV$ norm:
\begin{eqnarray} \label{Linear-Functional-tilde}
\tilde V(U,X) & = & \sum_{\substack{i \in \{0,\cdots, n\}\\ x_i \geq X}}^n \left( |\sigma_{i,1}| + |\sigma_{i,2}| \right) , \\
\label{Quadratic-Functional-tilde}
 \tilde Q(U,X) & = &
 \sum_{\substack{(x_i, \sigma_i) \\ x_i \geq X}}\Bigg(  |\sigma_{i}| \sum_{\substack{(x_j,\sigma_j) \hbox{\scriptsize \, approaching }\\ \hbox{\scriptsize \, and }x_j \geq X}} |\sigma_{j}| \Bigg).
\end{eqnarray}
When more than two fronts exit from an interaction point in the domain $D(0)$, we necessarily are in the situation of two interacting fronts of the same family and, thanks to Lemma~\ref{Lemma-Interaction}, their strength $\tilde \sigma$ and $\hat \sigma$ satisfy $|\tilde \sigma| |\hat \sigma| \geq h$ (for small $h$). But by \eqref{Decay-Q-inside}, this induces a decay of the functional $\tilde Q$ of $- h/2$ at least, while similar estimates as before show that for all $t \in (0,T_M)$,
\begin{equation*}
\tilde V(u_h(t),Mt) + c_0 \tilde Q(u_h(t),Mt) \leq e^{2 \gamma L}( V(0) + c_0 Q(0)).
\end{equation*}
(The only difference is that now, fronts can moreover leave the domain $D(0)$.) Therefore, there can be only a finite number of interactions from which more than two fronts exit.
Consequently, the number of fronts is finite in $D(0)$. \par
Accordingly, the number of fronts hitting the boundary $x = L$ during the time interval $[0,T_M]$ is finite. We now show that this implies that the number of newly created fronts in the set $[0,T_M] \times [0,L]$ is finite. Indeed, the functional $t\mapsto \tilde J(t) = \tilde V(u_h(t),0) + c_0 \tilde Q(u_h(t),0)$ satisfies the following properties:
\begin{itemize}
\item When no front hits the boundary and fronts do not meet in the domain, the functional $\tilde J$ is constant.
\item When a front hits the boundary and additional fronts (in finite number) are possibly created, the functional $\tilde J$ may increase. As this occurs a finite number of times, the increase of $\tilde J$ on $[0,T_M]$ is necessarily bounded.
\item When two fronts meet in the domain $(0,L)$, $\tilde J$ decreases. Besides, when more than two fronts issue from the interaction, $\tilde J$ decays from $-h/2$.
\end{itemize}
Consequently, there can be only a finite number of newly created fronts in the time interval $[0,T_M]$.
We can then iterate this argument on the time intervals of the form $[jT_M, (j+1)T_M]$ for $j \in \N$, from which we conclude $T_h^* = \infty$. 
\end{proof}
\subsection{Size of the rarefaction fronts}\label{Subsec-Rarefaction}
In this subsection, the term ``rarefaction front'' denotes any piecewise $C^1$ trajectory $t \mapsto x(t)$ corresponding to a rarefaction discontinuity. 
The goal of this section is to prove that the rarefaction fronts stay of strength ${\mathcal O}(h)$. This is central in the proof of the entropy inequality in the limit $h \to 0^+$. \par
\begin{lemma}
Under the assumption \eqref{Init-V-small-1}, there exists $C>0$ independent of $h$ and of the number of discontinuities in $u_{0,h}$ such that all the rarefactions fronts in the front-tracking approximations $u_{h}$ satisfy
\begin{equation} \label{Smallness-RS}
\sigma \leq C h.
\end{equation}
\end{lemma}
\begin{proof}
We argue similarly as in \cite[Section 7.3, item 5]{Bressan-Book-2000}. We thus consider a rarefaction front $(t, x(t))$ of strength $\sigma(t)>0$ and of family $k$ traveling in the domain. Let us say that it is created at some time $t_0 \geq 0$ with $\sigma(t_0) \leq h$ and then evolves according to the construction introduced in Section \ref{Sec-Construction} and then possibly ends on the right boundary. 
Due to the construction, a rarefaction front hitting the boundary can only generate new rarefaction fronts in the domain with a strength smaller than $h$, and we can therefore limit ourselves to rarefaction fronts until the first time they hit the boundary. As both speeds are strictly positive, we note that this exit time is necessarily limited by $t_0 + L/ c_* $. We denote this exit time by $t_0 + t_*$. \par
Now for such a rarefaction front there are two possibilities of interactions:
\begin{itemize}
\item It may interact with some fronts of the same family: in this case, as rarefaction fronts of the same family never interacts, it interacts necessarily with a shock of the same family. Then Glimm's interaction estimate in Lemma \ref{Lemma-Interaction}  shows that its strength becomes smaller.
\item It can meet fronts of the other family.
\end{itemize}
We will consequently focus on the second case. Our proof is based on an explicit bound on the maximal increase of the strength of a rarefaction front during a time interval of length $L/M$. A simple iteration argument implies then \eqref{Smallness-RS}. \par
We introduce $t_1$ such that
\begin{equation*}
	x(t_0) + M(t_1 -t_0)  = L,
\end{equation*}
where we recall that $M$ was defined in \eqref{DefM}. Next we define a quantity similar to $V$ (but following the rarefaction front that we consider) by
\begin{equation*}
V_{(x, \sigma)} (t) = \sum_{\substack{(x_i, \sigma_i) \hbox{\scriptsize \, approaching } (x(t), \sigma(t))\\ x_i > M (t - t_1)_+}} |\sigma_i|,
\end{equation*}
where the front $(x_i, \sigma_i)$ of family $k_i$ is said to be approaching the front $(x(t), \sigma(t))$ of family $k$ if one of the following cases occurs:
\begin{itemize}
	\item $k_i = k$ and $\sigma_i < 0$,
	\item $k_i < k$ and $x_i > x$,
	\item $k_i > k$ and $x_i < x$.
\end{itemize}
Concerning the fronts which may pass through the boundary, we introduce
$$
V_{(x, \sigma),b}(t) = \sum_{(x_i, \sigma_i) \hbox{\scriptsize \, with } x_i >x(t_0)+ M (t- t_0)} |\sigma_i|.
$$
Similarly, we introduce $Q_{(x,\sigma)}(t)$ and $Q_{(x,\sigma),b}(t)$
\begin{eqnarray*}
 Q_{(x,\sigma)} (t) & = &
 \sum_{\substack{ (x_i, \sigma_i) \hbox{\scriptsize \, approaching } (x(t), \sigma(t)) \\ x_i \geq M(t-t_1)_+}}  |\sigma_i| \Bigg(\sum_{x_j \geq M (t-t_1)_+ } |\sigma_{j}|    \Bigg), \\
 Q_{(x,\sigma),b} (t) & = &
 \sum_{\substack{ (x_i, \sigma_i)\\ x_i >x(t_0)+ M (t- t_0)}}  |\sigma_i|  \Bigg(\sum_{\substack{ (x_j,\sigma_j) \hbox{\scriptsize \, approaching } (x_i(t), \sigma_i(t))\\\hbox{\scriptsize \, with } x_j \geq x(t_0)+ M (t- t_0) }} |\sigma_{j}|   \Bigg).
\end{eqnarray*}
We now discuss the evolution of the quantities $\sigma$, $V_{(x,\sigma)}$, $V_{(x,\sigma),b}$, $Q_{(x,\sigma)}$, $Q_{(x,\sigma),b}$ as time evolves. We discuss the following four cases: 
\begin{itemize}
\item {\bf Case 1:} There are no interactions at time $t$.
\item {\bf Case 2:} At time $t$, there is an interaction inside $(0,L)$ that do not involve the rarefaction front $(x(t), \sigma(t))$.
\item {\bf Case 3:} At time $t$, there is an interaction (inside $(0,L)$) involving the rarefaction front $(x(t), \sigma(t))$.
\item {\bf Case 4:} At time $t$, a front hits the boundary. 
\end{itemize}
\noindent
{\bf Case 1.} When there is no interaction at time $t$, the quantities $\sigma$, $V_{(x,\sigma)}$, $V_{(x,\sigma),b}$, $Q_{(x,\sigma)}$, $Q_{(x,\sigma),b}$ cannot increase:
\begin{equation} \label{Case1}
\left\{ \begin{array}{l}
	\ds \sigma(t+) = \sigma(t-), \\
	\ds V_{(x,\sigma)}(t+) + c_0 Q_{(x,\sigma)}(t+) \leq V_{(x,\sigma)} (t-) + c_0 Q_{(x,\sigma)}(t-), \\
	\ds V_{(x,\sigma),b}(t+) + c_0 Q_{(x,\sigma),b}(t+) \leq V_{(x,\sigma),b}(t-) + c_0 Q_{(x,\sigma),b}(t-).
\end{array} \right.
\end{equation}		
\noindent
{\bf Case 2.} At times $t \in [t_0, t_0 +\min\{ L/M, t_*\}]$ corresponding to interactions inside $(0,L)$ that do not involve the rarefaction front $(t,x(t), \sigma(t))$, as in \cite[Section 7.3, item 5]{Bressan-Book-2000}, we obtain from Lemma \ref{Lemma-Interaction} that
\begin{equation} \label{Case2}
\left\{ \begin{array}{l}
	\ds \sigma(t+) = \sigma(t-), \\
	\ds V_{(x,\sigma)}(t+) + c_0 Q_{(x,\sigma)}(t+) \leq V_{(x,\sigma)} (t-) + c_0 Q_{(x,\sigma)}(t-), \\
	\ds V_{(x,\sigma),b}(t+) + c_0 Q_{(x,\sigma),b}(t+) \leq V_{(x,\sigma),b}(t-) + c_0 Q_{(x,\sigma),b}(t-).
\end{array} \right.
\end{equation}	
\noindent
{\bf Case 3.} At times $t \in [t_0, t_0 +\min\{ L/M, t_*\}]$ corresponding to an interaction inside $(0,L)$ involving the rarefaction front $(t,x(t),\sigma(t))$ with a front of strength $\sigma_\alpha$, from Lemma \ref{Lemma-Interaction},
\begin{equation} \label{Case3}
\left\{ \begin{array}{l}
	\ds  \sigma(t+) \leq \sigma (t-) + C_\delta |\sigma(t-)| |\sigma_\alpha|, \\
	\ds V_{(x,\sigma)}(t+) \leq V_{(x,\sigma)} (t-) - |\sigma_\alpha|, \\
	\ds  Q_{(x,\sigma)}(t+) \leq Q_{(x,\sigma)}(t-), \\
	\ds V_{(x,\sigma),b}(t+) + c_0 Q_{(x,\sigma),b}(t+) \leq V_{(x,\sigma),b}(t-) + c_0 Q_{(x,\sigma),b}(t-).
\end{array} \right.
\end{equation}
\noindent
{\bf Case 4.} At times $t \in [t_0, t_0 +\min\{ L/M, t_*\}]$ corresponding to a front hitting the boundary of strength $\hat \sigma$, from Lemma \ref{Lemma-Boundary},
\begin{equation} \label{Case4}
\left\{ \begin{array}{l}
\ds  \sigma(t+) = \sigma (t-) , \\
\ds V_{(x,\sigma)}(t+) + c_0 Q_{(x,\sigma)}(t+) \leq V_{(x,\sigma)} (t-) + c_0 Q_{(x,\sigma)}(t-) + (1- \varepsilon)|\hat \sigma|,  \\
\ds V_{(x,\sigma),b}(t+) + c_0 Q_{(x,\sigma),b}(t+) \leq V_{(x,\sigma),b}(t-) + c_0 Q_{(x,\sigma),b}(t-)-  |\hat \sigma|.
\end{array} \right.
\end{equation}	
In particular, relying on the estimates \eqref{Case1}--\eqref{Case4}, there exists some constant $C$ large enough such that 
$$
	t \mapsto \sigma(t) \exp\left( C (V_{(x,\sigma)}(t) + c_0 Q_{(x,\sigma)}(t) + V_{(x,\sigma),b}(t) + c_0 Q_{(x,\sigma),b}(t)) 	\right),
$$
is non-increasing on $[t_0,t_0 +\min\{ L/M, t_*\}]$. Introducing $J$ defined in \eqref{Def-J} and $\tilde V$, $\tilde Q$ defined in \eqref{Linear-Functional-tilde}--\eqref{Quadratic-Functional-tilde}, we immediately get that for all $t \in [t_0, t_0 +\min\{ L/M, t_*\}]$, 
$$
	V_{(x,\sigma)}(t) + c_0 Q_{(x,\sigma)}(t) + V_{(x,\sigma),b}(t) + c_0 Q_{(x,\sigma),b}(t) 
	\leq 
	\tilde V(u_h(t),0) + c_0\tilde Q(u_h(t),0) 
	\leq 
	e^{2 \gamma L} J(t).
$$
Therefore, the decay property \eqref{Exp-Decay-J} yields, for all $t \in [t_0, t_0 +\min\{ L/M, t_*\}]$,
$$
	\sigma(t) \leq \sigma(t_0) \exp( C J(0)).
$$
Of course, this estimate on the maximal amplification of $\sigma(t)$ stays valid in each interval of time of length $L/M$. As $t_*$ is necessarily bounded by $L/ c_*$, iterating this estimate at worst a finite number of times $\simeq M/c_*$, we get that for all $t$ during which the rarefaction exists, $\sigma(t)$ always satisfies	
$$
	\sigma(t) \leq  \sigma(t_0) \exp(C (V(0) + c_0 (V(0))^2) \leq C h,
$$
i.e. the last estimate \eqref{Smallness-RS} in Lemma \ref{Lemma-Interaction}.
\end{proof}
As a corollary of Lemma \ref{Lem-Lyapunov}, we have the following result.
\begin{corollary}[Bounds on the total variation] \label{cor-decay-bounds}
Under the setting of Theorem \ref{Thm-Main}, there exist positive constants $\varepsilon_2 \in (0, \varepsilon_1]$, $C>0$ and $\nu >0$ such that for any $h>0$, if $u_{0,h}$ is a piecewise constant function satisfying $TV^*_{[0,L]}(u_{0,h}) \leq \varepsilon_2$, if we define the approximate solution $u_h(t)$ of \eqref{System-U}--\eqref{BoundaryConditions} using the front tracking approximation explained above, for all $t \geq 0$,
\begin{equation} \label{Decay-BV-norm-h}
TV^*_{[0,L]}(u_h(t)) \leq C e^{- \nu t} TV^*_{[0,L]}(u_{0,h}).
\end{equation}
In particular, if we choose $u_{0,h}$ satisfying \eqref{Conv-u-0-h}, we get the uniform estimate:
\begin{equation} \label{Decay-BV-norm-h-2}
TV^*_{[0,L]}(u_h(t)) \leq C e^{- \nu t} TV^*_{[0,L]}(u_{0}).
\end{equation}
\end{corollary}
\begin{proof}
We recall that the quantities $TV^*_{[0,L]}(\cdot)$ and $V(\cdot)$ are equivalent for $u_h$ satisfying $\norm{u_h}_{L^\infty(0,L)} \leq \delta$, see \eqref{VTV*Equiv}. Therefore, taking $\varepsilon_2>0$ small enough to guarantee \eqref{Init-V-small-1}, the front tracking approximation $u_h$ is well-defined for all times according to Lemma \ref{Lemma-T-Infinity}, and we have the decay \eqref{Exp-Decay-J}, which yields
\begin{equation*}
V(t) \leq J(t) \quad \hbox{ and } \quad J(0) \leq V(0) + c_0 V(0)^2 \leq V(0) ( 1+ c_0),
\end{equation*}
where we assumed in the last identity that $V(0) \leq 1$, which can be done by taking $\varepsilon_2>0$ small enough. This obviously implies \eqref{Decay-BV-norm-h}.
\end{proof}
%
%
%
%
%
%
\subsection{Lipschitz in time estimates}\label{Subsec-TimeLips}
In order to pass to the limit in the boundary conditions \eqref{BoundaryConditions}, we will also estimate the norm of the approximations in $W^{1, \infty}(0,L; L^1_{loc}(0,\infty))$.
Our strategy is inspired by sideways energy estimates which are made possible by the fact that the velocities are bounded from below. These types of energy estimates consist in exchanging the role played by the time and space variables. 
\par
The main result of this section is the following one.
\begin{lemma} \label{Lemma-Lips-Boundary}
Under the assumptions of Lemma \ref{Lem-Lyapunov}, the approximate solution $u_h$ constructed in Section \ref{Sec-Construction} satisfies the following property: for all $T>0$, there exists $C_{T, \varepsilon_0}$ with
\begin{equation} \label{Est-Ortho}		
\sup_{x \in [0,L]} TV(u_h(\cdot, x)_{|t \in (0,T)}) \leq C_{T, \varepsilon_0}.
\end{equation}
\end{lemma}
\begin{proof}
Our main argument here consists in considering the front-tracking approximations $u_{h}$ with $x$ as a time variable, which is made simple by \eqref{PositiveVel}. 
We thus introduce the functionals
\begin{equation}
V_\perp(x,T) = \sum_{\hbox{\scriptsize fronts at } (x,t) \hbox{\scriptsize\, with } t < T} |\sigma|,
\end{equation}
and
\begin{equation}
Q_\perp(x,T) = \sum_{\hbox{\scriptsize fronts } (x,t_i) \hbox{\scriptsize\, with } t_i < T} |\sigma_i| \Bigg(\sum_{\substack{\hbox{\scriptsize fronts at } (x,t_j) \hbox{\scriptsize\, with } t_j < T\\ (x,t_j) \hbox{\scriptsize\, approaching } (x, t_i)} } |\sigma_j| \Bigg),
\end{equation}
where here a front at $(x,t_j)$ with strength $\sigma_j$ and family $k_j$ is said to be approaching the front at $(x, t_i)$ with strength $\sigma_i$ and family $k_i$ if
\begin{itemize}
	\item $t_j < t_i$, $k_i = k_j$ and $\sigma_i$ or $\sigma_j$ is negative,
	\item $t_j < t_i$ and $(k_i , k_j) = (1,2)$.
\end{itemize}
Using the functionals in \eqref{Linear-Functional-tilde}--\eqref{Quadratic-Functional-tilde}, one can check that the map
$$
	x \mapsto \tilde V(u_h(x/M), x) + c_0 \tilde Q(u_h(x/M),x) + V_\perp(L,x/M)
$$
decays. This implies in particular that 
$$
	V_\perp(L,L/M) \leq \tilde V(u_{0,h},0) + c_0 \tilde Q(u_{0,h},0) \leq e^{2 \gamma L} J(0).
$$
Thanks to the conditions \eqref{Assumption-On-K},
\begin{equation} \label{Initiation-X}		
	V_\perp(0,L/M) \leq V_\perp(L,L/M) \leq e^{2 \gamma L} J(0).
\end{equation}
Now, assume we have a bound on $V_\perp(0,t_0)$ for some $t_0 >0$. Then we claim that we can have the estimate
\begin{multline} \label{iteration-formula-X}		
\sup_{x \in [0,L]} \{ V_\perp(x, t_0 + x/M) + c_0 Q_\perp(x, t_0 + x/M) \} \\
\leq (V_\perp(0,t_0) + c_0 (V_\perp(0,t_0))^2) + e^{2\gamma L} \left(V(0) + c_0 Q(0)\right).
\end{multline}
In order to show this, it is enough to consider the evolution of
$$
	x \mapsto V_\perp(x, t_0 + x/M) + c_0 Q_\perp(x, t_0 + x/M).
$$
Indeed, according to Glimm's interaction estimate in Lemma~\ref{Lemma-Interaction}, if two fronts meet in $(0,t_0 +x/M) \times \{ x\}$, this functional decays. Fronts may also arise at $t=0$ from the boundary $(0,x)$, in which case it is measured in $J(0)$. Finally, fronts can leave the domain $(0, t_0 + x/M) \times \{x\}$ from $t = t_0 + x/M$ (remember $M$ is given by \eqref{DefM}). \par
Using \eqref{Initiation-X} and iterating \eqref{iteration-formula-X}, we obtain that for $V(0) \leq \varepsilon_2$, for all $T>0$, there exists $C_{T, \varepsilon_0}$ with
$$
	\sup_{x} V_\perp(x, T) \leq C_{T, \varepsilon_0}.
$$
This concludes the proof of Lemma \ref{Lemma-Lips-Boundary} as $V_\perp(x, T)$ is equivalent to the $TV$-semi norm of $u_h(\cdot, x)$.
\end{proof}
%
%
%
%
%
%
%
%
\subsection{Passing to the limit: end of the proof of Theorem \ref{Thm-Main}}\label{Subsec-Limit}

Let $u_0 \in BV(0,L)$, and consider a sequence of approximations $u_{0,h}$ converging to $u_0$  strongly in $L^1(0,L)$ and a.e. in $(0,L)$ as $h \to 0$ with $TV^*_{[0,L]} (u_{0,h})\leq 2 TV^*_{[0,L]} (u_{0})$ such that $u_{0,h}$ is piecewise constant, and consider the corresponding sequence $u_h$ constructed in Section \ref{Sec-Construction}. \par
Lemma \ref{Lem-Lyapunov} and Corollary \ref{cor-decay-bounds} show that $u_h$ is bounded in $L^\infty(0,\infty; BV(0,L))$. Besides, it is easy to check that, thanks to finite speed of propagation, there exists a constant $C>0$ such that for all $t_1> t_0 >0$,
\begin{equation} \label{Est-W-1-infty-L-1}
\norm{u_h(t_1) -u_h(t_0)}_{L^1(0,L)} \leq C |t_1 -t_0|\max_{t \in [t_0,t_1]} TV^*_{[0,L]}(u_h(t)) \leq C |t_1 -t_0| TV^*_{[0,L]}(u_0) .
\end{equation}
Therefore, one can use Helly's theorem and a diagonal extraction argument (see \cite[Theorem 2.4]{Bressan-Book-2000}) and obtain a limit function $u \in L^1_{loc}(0, \infty; BV(0,L))$ such that, up to a subsequence still denoted in the same way for simplicity, $u_h$ strongly converges as $h\to 0$ to $u$ in $L^1_{loc}((0,\infty)\times (0,L))$. Besides, $u \in L^\infty(0,\infty; BV(0,L))\cap W^{1, \infty}(0,\infty; L^1(0,L))$ and satisfies
\begin{equation} \nonumber
\norm{u(t_1) -u(t_0)}_{L^1(0,L)} \leq C |t_1 -t_0| TV^*_{[0,L]}(u_0), \qquad t_1 >t_0 >0.
\end{equation}
In particular, since $u_h(0) = u_{0,h}$ strongly converges to $u_0$ in $L^1(0,L)$, we immediately get $u(0) = u_0$. \par
Furthermore, using the semi-continuity of the $TV^*_{[0,L]}$ norm and passing to the limit in \eqref{Decay-BV-norm-h}, we obtain, for all $t \geq 0$,
\begin{equation}
	\label{Decay-BV-norm}
	TV^*_{[0,L]}(u(t)) \leq C e^{- \nu t} TV^*_{[0,L]}(u_0), 
\end{equation}
which proves \eqref{Exp-Stabilization-Estimate}.
\par
To derive that $u$ necessarily is a weak entropy solution of \eqref{System-U} in $(0,\infty) \times (0,L)$, we argue as in \cite[Section 7.4]{Bressan-Book-2000}. That step mainly consists in measuring the errors done in approximating rarefaction waves by rarefaction fronts. This relies on the estimate \eqref{Smallness-RS} on the size of the rarefaction fronts. Details of the proof are left to the reader.\par
It remains to prove that the boundary conditions \eqref{BoundaryConditions} are satisfied. According to Lemma \ref{Lemma-Lips-Boundary}, for all $T>0$, the approximate solutions $u_h$ are uniformly bounded in $L^\infty(0,L; BV(0,T))$. As the velocities are strictly positive, it follows, similarly as in \eqref{Est-W-1-infty-L-1}, that for all $T>0$, $u_h$ are uniformly bounded in $W^{1, \infty}(0,L; L^1(0,T))$. Accordingly, using again \cite[Theorem 2.4]{Bressan-Book-2000}, $u$ satisfies the boundary conditions \eqref{BoundaryConditions} almost everywhere.

\bibliographystyle{plain}

\end{document}